\documentclass[11pt]{amsart}
\usepackage{amssymb,amsmath,amsthm,amsfonts,amsopn,url,bbm,MnSymbol,hyperref, color}
\usepackage[all]{xy}
\theoremstyle{plain}
\newtheorem{thm}{Theorem}[section]
\newtheorem{prop}[thm]{Proposition}
\newtheorem{lemma}[thm]{Lemma}

\theoremstyle{definition}
\newtheorem{defn}[thm]{Definition}
\newtheorem*{defn*}{Definition}
\newtheorem{question*}{Question}

\newtheorem*{example*}{Example}
\newtheorem{rmk}[thm]{Remark}
\newtheorem*{rmk*}{Remark}
\newcommand{\field}[1]{\mathbbm{#1}}
\newcommand{\N}{\field{N}}

\newcommand{\F}{\field{F}}
\newcommand{\ideal}[1]{\mathfrak{#1}}
\newcommand{\m}{\ideal{m}}

\newcommand{\p}{\ideal{p}}
\newcommand{\q}{\ideal{q}}

\newcommand{\func}[1]{\mathrm{#1} \,}

\newcommand{\Spec}{\func{Spec}}

\newcommand{\arrow}[1]{\stackrel{#1}{\rightarrow}}

\newcommand{\ra}{\rightarrow}

\DeclareMathOperator{\ann}{ann}

\newcommand{\cM}{\mathcal{M}}
\newcommand{\cL}{\mathcal{L}}

\DeclareMathOperator{\Supp}{Supp}

\DeclareMathOperator{\Prl}{Prl}
\DeclareMathOperator{\md}{mod}
\DeclareMathOperator{\Mod}{Mod}
\DeclareMathOperator{\diam}{diam}
\DeclareMathOperator{\gir}{girth}
\DeclareMathOperator{\clq}{clique}
\DeclareMathOperator{\chro}{\chi}
\newcommand{\Max}[1]{\Omega(#1)}
\DeclareMathOperator{\USpec}{Alex}
\DeclareMathOperator{\NAK}{NAK}
\DeclareMathOperator{\Cyc}{Cyc}
\DeclareMathOperator{\Id}{Id}
\DeclareMathOperator{\FId}{FId}
\newcommand{\ver}[1]{\mathrm{v}({#1})}
\DeclareMathOperator{\Jac}{Jac}
\newcommand{\AG}{{\mathbb{AG}}}
\DeclareMathOperator{\CId}{\omega-Id}

\author{Neil Epstein}
\address{Universit\"at Osnabr\"uck \\ 
Institut f\"ur Mathematik \\ 
49069 Osnabr\"uck \\ Germany}
 \email{nepstein@uni-osnabrueck.de}

\author{Peyman Nasehpour}
\address{Department of Engineering Science\\
Faculty of Engineering\\
University of Tehran\\
Tehran, Iran}
\email{nasehpour@gmail.com}

\title{Zero-divisor graphs of nilpotent-free semigroups}


\date{May 21, 2012}
\begin{document}
\begin{abstract}
 We find strong relationships between the zero-divisor graphs of apparently disparate kinds of nilpotent-free semigroups by introducing the notion of an \emph{Armendariz map} between such semigroups, which preserves many graph-theoretic invariants.  We use it to give relationships between the zero-divisor graph of a ring, a polynomial ring, and the annihilating-ideal graph.  Then we give relationships between the zero-divisor graphs of certain topological spaces (so-called pearled spaces), prime spectra, maximal spectra, tensor-product semigroups, and the semigroup of ideals under addition, obtaining surprisingly strong structure theorems relating ring-theoretic and topological properties to graph-theoretic invariants of the corresponding graphs.
\end{abstract}
 \thanks{The first named author was partially supported by a grant from the DFG (German Research Foundation)}

\maketitle
\tableofcontents

\section{Introduction}

The original motivation for this work was the following: For a commutative ring $R$ with unity, let $({}_R\md, \otimes_R)$ be the semigroup of isomorphism classes of finite $R$-modules, with semigroup operation given by tensor product.  What does the zero-divisor graph of this semigroup look like, from the point of view of the ring-theoretic properties of $R$?  In investigating this question, we encountered connections with  topology, semigroup theory, lattice theory, multiplicative ideal theory, and other areas.

\begin{defn}
We define a (multiplicatively written) semigroup $(S,\cdot)$ to be a \emph{semigroup with zero} if it is commutative and has an absorbing element $0$.  We say $S$ is  \emph{nilpotent-free} if for any element $x\in S$ such that $x^n=0$ for some positive integer $n$, it follows that $x=0$.
\end{defn}

A version of the notion of zero-divisor graph of a \emph{ring} was first given by Beck \cite{Be-color}, with the by-now standard definition given by Anderson and Livingston \cite{AnLi-zd} some years later.  However, it is the more general notion of the zero-divisor graph of a \emph{semigroup with zero} that will be basic to the present paper.

\begin{defn}\cite{DMS-zdsemi}
Let $(S,\cdot)$ be a semigroup with zero.  Then the \emph{zero-divisor graph} $\Gamma(S)=\Gamma(S,\cdot)$ of $S$ is the graph whose vertex set $\ver{\Gamma(S)}$ is given by all the nonzero zero-divisors of $S$, such that $\{s,t\}$ is an edge precisely when $s\cdot t=0$.
\end{defn}

Then the zero-divisor graph of a ring $R$ from \cite{AnLi-zd} is $\Gamma(R,\cdot)$.

In our investigation, we found the notion of \emph{Armendariz map} (see Definition~\ref{def:Arm}) to be central.  It is a kind of generalization of the notion of isomorphism, preserving many invariants of the corresponding zero-divisor graphs, which nonetheless leaves quite a bit of room for relating different sorts of semigroups.  In particular, for any commutative ring $R$, we have the following commutative diagrams, where all the horizontal arrows are Armendariz maps.

\[\xymatrix{
&({}_R\NAK, \otimes_R) \ar[r]^{\Supp} & (\sigma(\USpec R), \cap) \ar[r]^{\cap \Max R} & (2^{\Max R}, \cap)\\
&({}_R\md, \otimes_R) \ar[r]^{\Supp} \ar@{^(->}[u] & (\sigma(\Spec R), \cap)\ar@{^(->}[u] \ar[r]^{\cap \Max R} & (\sigma(\Max R), \cap)\ar@{^(->}[u]\\
(\Id{R}, +) \ar[r]^{\cong} &({}_R\Cyc, \otimes_R) \ar[r]^{\Supp}  \ar@{^(->}[u] &  (\sigma(\Spec R), \cap) \ar@{=}[u] \ar[r]^{\cap \Max R} & (\sigma(\Max R), \cap)\ar@{=}[u]
}\]
and
\[\xymatrix{
(R[\![X]\!], \cdot) \ar[r]^{c'} & (\CId{R}, \cdot)\\
(R[X], \cdot) \ar@{^(->}[u] \ar[r]^c & (\FId{R}, \cdot)\ar@{^(->}[u]
}\]

Some familiar objects in these diagrams include \begin{itemize}
\item $({}_R\md, \otimes_R)$, whose elements are the isomorphism classes of the finitely generated $R$-modules,
\item $\sigma(\Spec R)$, the Zariski-closed subsets of the prime spectrum,
\item $\sigma(\Max{R})$, the Zariski-closed subsets of the maximal spectrum,
\item $\FId{R}$ (resp. $\CId{R}$, resp. $\Id{R}$), the finitely generated (resp. countably generated, resp. arbitrary) ideals of $R$,
\item $\Supp$, the support of a module,
\item $c$, the \emph{content map}, which sends each polynomial to the ideal generated by its coefficients. 
\end{itemize}
 
Here is a brief sketch of the contents of our paper:

In \S\ref{sec:Arm}, we introduce the unifying concept of our paper (the \emph{Armendariz map}) along with the theorem that shows that many graph-theoretic invariants are preserved or at least controlled by such maps.  In \S\ref{sec:eq}, we show how this concept is related to the \emph{equivalence class} graph of a ring that has been investigated by several authors.  In \S\ref{sec:content}, we give the first examples of Armendariz maps, between polynomial (resp. power series) extensions and the semigroup of finitely- (resp. countably-) generated ideals.  In the process, we provide an alternate proof of a conjecture of Behboodi and Rakeei.  In \S\ref{sec:lattice}, we consider among other things the semigroup whose elements are the closed subsets of a $T_1$ topological space $Y$, with product given by intersection, showing intimate connections between familiar properties of the space $Y$ with graph-theoretic invariants of the corresponding zero-divisor graph.  In \S\ref{sec:prl}, we explore the properties of so-called \emph{pearled} topological spaces, each of which may be associated via an Armendariz map with a $T_1$ space.  In \S\ref{sec:spec}, we look at two topologies on the prime ideal spectrum of a ring and show that certain properties (e.g. the number of maximal ideals and the primality of the Jacobson radical) may be detected by graph-theoretic invariants of the corresponding zero-divisor graphs.  Using yet another Armendariz map in \S\ref{sec:tensor}, we answer our original question on the zero-divisor graph of finitely generated (resp \emph{locally Nakayama}) modules up to isomorphism under tensor product, showing clean connections between ring-theoretic properties (e.g. the number of maximal ideals, the primality of the Jacobson radical) with graph-theoretic invariants of such graphs.  Finally, in \S\ref{sec:comax} we compute graph-theoretic invariants of the semigroup of ideals under \emph{addition}, which is closely related to so-called comaximal graphs in the literature.

The moral of our story is that one need not study zero-divisor graphs of one particular kind of semigroup in isolation, as one may be able to obtain information about many other kinds of semigroups if one knows something about their relationships with one another.  In particular, even if one's primary field of interest is commutative ring theory (as is the case with the authors), the study of zero-divisor graphs only of commutative rings is an unnecessary restriction.

\section{The fundamental theorem on Armendariz maps}\label{sec:Arm}
In this section, we introduce our main tool for translating graph-theoretic information from one nilpotent-free semigroup to another:

\begin{defn}\label{def:Arm}
Let $g: S \ra T$ be a set map between nilpotent-free semigroups.  We say that $g$ is an \emph{Armendariz} map if it satisfies the following three conditions: \begin{enumerate}
\item $g$ is surjective.
\item For any $s\in S$, $s=0$ if and only if $g(s) = 0$.
\item For any $s, s' \in S$, we have $ss'= 0$ if and only if $g(s) g(s') = 0$.
\end{enumerate}
\end{defn}

\begin{rmk}
This definition may seem unnatural to some readers.  After all, shouldn't a map between semigroups be required to be a semigroup homomorphism?  In most examples, this seems to be true; accordingly, the reader should note that if $g$ is a homomorphism of nilpotent-free semigroups, it is an Armendariz map if and only if it is surjective and \emph{kernel-free} (i.e. $g(s) = 0 \implies s=0$).   In two important cases, namely those of \emph{(generalized) content}, we have an Armendariz map that is usually not a homomorphism (see e.g. Lemma~\ref{lem:Arm-not-Gauss}; content is Armendariz when the ring is reduced, but it is only a homomorphism when the ring is Gaussian) but nevertheless yields the properties we want.  Hence we operate at this level of generality.  However, the reader may spend most of the paper thinking of Armendariz maps as being surjective kernel-free homomorphisms.

This brings up another point.  Namely, if $S$ and $T$ are reduced \emph{rings}, then a kernel-free surjective ring homomorphism is just an isomorphism.  The idea of an Armendariz map is that it approximates the properties of an isomorphism closely enough that much of the structure reflected in a zero-divisor graph is preserved.
\end{rmk}

Recall the following: 
\begin{defn}\cite{DMS-zdsemi}
Let $S$ be a semigroup with zero.  Then the \emph{zero-divisor graph} $\Gamma(S)$  of $S$ is the simple graph whose vertices are the non-zero zero-divisors of $S$, such that two distinct such elements $a,b$ form an edge precisely when $ab=0$
\end{defn}

The invariants of graphs that we care about in this paper are the following:
\begin{defn}
For a graph $G$, the \emph{diameter} $\diam G$ is defined to be the supremum of the distances between any pair of vertices.  The \emph{girth} $\gir G$ is the length of the shortest cycle, or if there are no cycles, we say $\gir G = \infty$.  The \emph{clique number} $\clq G$ of $G$ is the supremum of the cardinalities of subsets $S$ of vertices of $G$ such that the induced subgraph of $G$ on $S$ is complete.  (Such subsets are called \emph{cliques}.) The \emph{chromatic number} $\chro(G)$ of $G$ is the infimum of the cardinalities of sets $A$ such that there exists a set map $c: \ver{G} \ra A$ such that whenever $v,w \in \ver{G}$ determine an edge, we have $c(v) \neq c(w)$.  (Such a map is called an \emph{$A$-coloring} of $G$.)
\end{defn}

Some important properties are collected below:
\begin{prop}
For any semigroup $S$ with zero, $\Gamma(S)$ is connected  \cite[Theorem 1.2]{DMS-zdsemi}, $\diam {\Gamma(S)} \leq 3$ \cite[Theorem 1.2]{DMS-zdsemi}, and $\gir{\Gamma(S)} \in \{3, 4, \infty\}$ \cite[Theorem 1.5]{DMS-zdsemi}.
\end{prop}

\begin{thm}\label{thm:Armendariz}
Let $\rho: S \ra T$ be an Armendariz map of nilpotent-free semigroups. \begin{enumerate}
\item\label{it:diam} If $\diam \Gamma(T) \neq 1$, then $\diam \Gamma(S) = \diam \Gamma(T)$.
\item\label{it:diam1} If $\diam \Gamma(T) =1$, then \[
\diam \Gamma(S) = \begin{cases}
1, &\text{if the induced map } \Gamma(S) \ra \Gamma(T) \text{ is bijective,}\\
2, &\text{otherwise.}
\end{cases}
\]
\item\label{it:girfin} If $\gir \Gamma(T) < \infty$, then $\gir \Gamma(S) = \gir \Gamma(T)$.
\item\label{it:girinf} If $\gir \Gamma(T) = \infty$, then $\gir \Gamma(S) \in \{4, \infty\}$ (and one usually expects $4$).
\item\label{it:clq} $\clq \Gamma(S) = \clq \Gamma(T)$.
\item\label{it:chro} $\chro(\Gamma(S)) = \chro(\Gamma(T))$.
\end{enumerate}
\end{thm}

\begin{proof}
First note that there is an induced map $\rho_*: \Gamma(S) \ra \Gamma(T)$ that takes vertices to vertices, takes edges to edges, and is surjective on both.  Moreover, given an edge $\{t,t'\}$ in $\Gamma(T)$, then for \emph{any} preimages $s$, $s'$ of $t$, $t'$ respectively, $\{s, s'\}$ is then an edge in $\Gamma(S)$.  These conditions are essentially the translation into zero-divisor graph language of what it \emph{means} to have an Armendariz map.

\begin{proof}[Proof of (\ref{it:diam})] If $\diam \Gamma(T) = 0$, then $\Gamma(T)$ is empty (since $T$ is nilpotent-free), which means that $\Gamma(S)$ is empty as well, so that $\diam \Gamma(S) = 0$.

Now suppose $\diam \Gamma(T) = 2$, and let $t, t' \in \ver{\Gamma(T)}$ such that $d(t, t') = 2$.  Let $s, s' \in \ver{\Gamma(S)}$ such that $\rho(s) = t$ and $\rho(s') = t'$.  Then $ss' \neq 0$, so $d(s, s') \geq 2$.  Hence $\diam \Gamma(S) \geq 2$.  On the other hand, take any pair of distinct elements $s, s' \in \ver {\Gamma(S)}$, and suppose $ss' \neq 0$.  Then $\rho(s) \rho(s') \neq 0$, so since $\diam \Gamma(T) = 2$, there is some $t'' \in \ver{\Gamma(T)}$ such that $\rho(s) t'' = \rho(s') t'' = 0$.  Picking $s'' \in S$ such that $\rho(s'') = t''$, it follows that $ss'' = s's'' = 0$.  Moreover, all of $s$, $s'$, $s''$ must be distinct since $S$ is nilpotent-free, so this makes a path of length 2 in $\Gamma(S)$.  It follows that $\diam \Gamma(S) = 2$.

Finally suppose that $\diam \Gamma(T) = 3$.  Since we know that $\diam \Gamma(S) \leq 3$, we must only show that $\diam \Gamma(S) >2$.  So let $t, t' \in \ver{\Gamma(T)}$ such that $d(t,t') = 3$.  Let $s, s' \in \ver{\Gamma(S)}$ with $\rho(s) = t$ and $\rho(s') = t'$.  Since $tt' \neq 0$, we have $ss' \neq 0$, so $d(s,s') \geq 2$.  If $d(s,s') = 2$, then there is some $s'' \in \ver{\Gamma(S)}$ such that $ss'' = s's'' = 0$.  But then $t \rho(s'') = t'\rho(s'') = 0$, so that $d(t, t') \leq 2$, contradicting the fact that $d(t,t') = 3$.  Thus there is no such $s''$, whence $d(s,s') \geq 3$, so that $d(s,s')=3$, whence $\diam \Gamma(S) \geq 3$, hence $=3$.
\end{proof}

\begin{proof}[Proof of (\ref{it:diam1})]
If the induced map is bijective, then the zero-divisor graphs are in fact isomorphic, so that in particular their diameters must coincide.

If not, let $t \in \ver{\Gamma(T)}$ such that there exist distinct $r,s \in \ver{\Gamma(S)}$ with $\rho(r) = \rho(s) = t$.  If $rs = 0$, then $t^2=\rho(r) \rho(s) = 0$, which would contradict the non-nilpotency condition of $T$.  Hence $\diam \Gamma(S) \geq 2$.  On the other hand, take any $a, b \in \ver{\Gamma(S)}$.  If $\rho(a) \neq \rho(b)$, then since $\diam \Gamma(T) = 1$, we have $d(\rho(a), \rho(b)) = 1$, which means there is an edge between them, whence $\rho(a) \rho(b) = 0$.  But since the map is Armendariz, $ab = 0$, so that $d(a,b) = 1$.  If however $\rho(a) = \rho(b)$ then the earlier argument shows that $ab \neq 0$.  But there is some $c\in \ver{\Gamma(S)}$ such that $ac = 0$.  Then $\rho(a) \rho(c) = \rho(b) \rho(c) = 0$, which then implies that $bc = 0$.  Thus, we have a path from $a$ to $c$ to $b$, which shows that $d(a,b) = 2$.  Hence, $\diam \Gamma(S) = 2$.
\end{proof}

\begin{proof}[Proof of (\ref{it:girfin})]
First suppose $\gir \Gamma(T) = 3$.  Then there exist distinct elements $t, t', t'' \in \ver{\Gamma(T)}$ such that $tt' = tt'' = t't'' = 0$.  Since $\rho$ is surjective, then, there exist distinct elements $s, s', s'' \in S$ such that $\rho(s) = t$, $\rho(s') = t'$, and $\rho(s'') = t''$.  Then by the Armendariz property, we have $ss' = ss'' = s's'' = 0$, which makes a 3-cycle in $\Gamma(S)$, so that $\gir \Gamma(S) = 3$.

If $\gir \Gamma(T) = 4$, then preimages of the vertices in a 4-cycle of $\Gamma(T)$ make a 4-cycle in $\Gamma(S)$ (as with the girth 3 case above), so we need only show that $\Gamma(S)$ has no 3-cycles.  By contradiction, suppose $s, s', s'' \in \ver{\Gamma(S)}$ form a 3-cycle.  Then by the Armendariz property, $\rho(s) \rho(s') = \rho(s) \rho(s'') = \rho(s') \rho(s'') = 0$.  But since $\Gamma(T)$ has no 3-cycles, it follows that these elements cannot all be distinct, so that without loss of generality $\rho(s) = \rho(s')$.  But then $\rho(s)^2 = 0$, so that since $T$ has no nilpotents, $\rho(s) = 0$, so that $s=0$, contradicting the fact that $s\in \ver{\Gamma(S)}$.  Thus, $\gir \Gamma(S) = 4$.
\end{proof}

\begin{proof}[Proof of (\ref{it:girinf})]
We need only show that $\gir \Gamma(S) \neq 3$.  But the proof of this is exactly the same as when we assumed that $\gir \Gamma(T) = 4$.

As for the informal comment about ``expectation'', we note here two common ways it can happen that $\Gamma(S)$ can have girth 4.  The first is to have $t, t' \in \ver{\Gamma(T)}$ such that $tt' = 0$ in such a way that neither of the sets $\rho^{-1}(\{t\})$,  $\rho^{-1}(\{t'\})$ are singletons.  The second is to have a vertex $t \in \ver{\Gamma(T)}$  such that there exist $t', t'' \in \ver{\Gamma(T)}$ with $tt' = tt'' = 0$ and such that $\rho^{-1}(\{t\})$ is not a singleton.
\end{proof}

\begin{proof}[Proof of (\ref{it:clq})]
Let $B$ be a set of vertices in $\Gamma(T)$ that forms a clique.  For each $b\in B$, let $c_b \in S$ such that $\rho(c_b) = b$.  Then each such $c_b \in \ver {\Gamma(S)}$.  Let $A := \{c_b \mid b \in B\}$.  It is easy to see that $A$ forms a clique in $\Gamma(S)$ of the same cardinality as $B$.  Hence $\clq \Gamma(S) \geq \clq \Gamma(T)$.

For the reverse inequality, let $A$ be a set of vertices in $\Gamma(S)$ that forms a clique.  Then for any pair $a, a'$ of distinct elements of $A$, we have $aa' = 0$, so that $\rho(a) \rho(a') = 0$.  Since $T$ has no nilpotent elements and $\rho$ sends no nonzero elements to zero, it follows that $\rho(a) \neq \rho(a')$.  It follows that $B := \{\rho(a) \mid a \in A\}$ forms a clique in $\Gamma(T)$ of the same cardinality as $A$.  Hence $\clq \Gamma(S) \leq \clq \Gamma(T)$.
\end{proof}

\begin{proof}[Proof of (\ref{it:chro})]
Let $A$ be a set and let $c: \ver{\Gamma(T)} \ra A$ be a coloring of $\Gamma(T)$.  Then $c \circ \rho_*$ is a coloring of $\Gamma(S)$.  To see this, take any edge $\{s,s'\}$ in $\Gamma(S)$.  Then $ss' = 0$, so $\rho(s) \rho(s') = 0$, so that $\rho(s) \neq \rho(s')$ (since $\rho$ takes no nonzero elements to zero and $T$ has no nilpotents), whence $\{\rho(s), \rho(s')\}$ is an edge of $\Gamma(T)$.  Thus \[
(c \circ \rho_*)(s) = c(\rho(s)) \neq c(\rho(s')) = (c \circ \rho_*)(s').
\]
Thus, $\chro(\Gamma(S)) \leq \chro(\Gamma(T))$.

For the reverse inequality, let $A$ be a set and let $c: \ver{\Gamma(S)} \ra A$ be a coloring of $\Gamma(S)$.  We define a set map $\alpha: \ver{\Gamma(T)} \ra \ver{\Gamma(S)}$ as follows:  For each $t \in \ver{\Gamma(T)}$, choose an element $s\in \ver{\Gamma(S)}$ such that $\rho(s) = t$; then let $\alpha(t) = s$.  Then $(c \circ \alpha)$ is a coloring of $\Gamma(T)$.  To see this, let $\{t, t'\}$ be an edge in $\Gamma(T)$.  Then we have $\rho(\alpha(t)) \rho(\alpha(t')) = tt' = 0$, so that by the Armendariz property, $\alpha(t)\alpha(t') = 0$, so that $\{\alpha(t), \alpha(t')\}$ is an edge in $\Gamma(S)$.  Thus \[
(c \circ \alpha)(t) = c(\alpha(t)) \neq c(\alpha(t')) = (c \circ \alpha)(t').
\]
Thus, $\chro(\Gamma(S)) \geq \chro(\Gamma(T))$, completing the proof.
\end{proof}

\end{proof}

\begin{rmk}\label{rmk:nonreduced}
If one wanted to expand this theorem to the case of semigroups with nilpotent elements, another condition would be necessary.   Namely, one could require along with the conditions in Definition~\ref{def:Arm} that if $s\neq s' \in S$ and $ss' = 0$, then $\rho(s) \neq \rho(s')$ (which, as we have seen, is automatic for Armendariz maps between nilpotent-free semigroups).  Indeed, given such a map, the entire proof of Theorem~\ref{thm:Armendariz} goes through!  However, in the cases of which we are aware, including the equivalence class quotient (cf. \S\ref{sec:eq}), this appears to be an unnatural condition to impose externally.
\end{rmk}

 \section{The equivalence class quotient}\label{sec:eq}
The connoisseur of zero-divisor graph theory may have noticed that our notion of Armendariz map bears some similarity to a concept that was first introduced by Mulay \cite{Mu-cszd} and was the main topic of the paper \cite{SpWi-eqzd} of Spiroff and Wickham.  Namely, if $R$ is a commutative ring, then for $x,y \in R$, these authors introduce an equivalence relation $\sim$ so that $x\sim y \iff \ann(x) = \ann(y)$.  Then $\Gamma_E(R)$ is the graph whose vertices are the \emph{equivalence classes} of the nonzero zero-divisors of $R$, drawing an edge from $[x]$ to $[y]$ if and only if $xy=0$.

However, one can see this instead as an operation on \emph{semigroups}.  Namely, if $S$ is a nilpotent-free semigroup, define an equivalence relation on $S$ as above and denote the set of equivalence classes as $E(S) = \{[s] \mid s\in S\}$.

\begin{prop}
$E(S)$ is a nilpotent-free semigroup, and the set map $e_S: S \ra E(S)$ sending $s \mapsto [s]$ is a kernel-free surjective semigroup homomorphism (hence an Armendariz map).
\end{prop}

\begin{proof}
Without loss of generality we may assume that $S \neq 0$.

We must first show that the multiplication $( [s]\cdot [t] := [st])$ in $E(S)$ is well-defined.  That is, given $s,t, s', t' \in S$  such that $s \sim s'$ and $t \sim t'$, we must show that $st \sim s't'$.  For any $a\in \ann(st)$, we have $ast = (as)t = 0$, so that $as \in \ann(t) = \ann(t')$, so that $0=ast' = (at')s$, whence $at' \in \ann(s) = \ann(s')$, so that $as't' = 0$.   Thus, $\ann(st) \subseteq \ann(s't')$, and by symmetry the opposite inclusion also holds.

Surjectivity is obvious, as is the fact that $e_S$ is a semigroup homomorphism.

Next, suppose that $[s] = 0$.  This means that $\ann(s) = \ann(0) = S$, and in particular $s\in \ann s$, so that $s^2 = 0$.  But since $S$ is nilpotent-free, $s=0$.  Thus, $S$ is kernel-free.  It also follows that $E(S)$ is nilpotent-free, since if $[s]^n = [s^n]=0$, then $s^n=0$, whence $s=0$.
\end{proof}

Indeed, more is true.  Namely, $E(S)$ is the \emph{final} Armendariz image of $S$, in the following sense:

\begin{thm}
Let $S$, $T$ be nilpotent-free semigroups, and let $g: S \ra T$ be an Armendariz map.  Then there is a unique Armendariz map $h: T \ra E(S)$ such that $h \circ g = e_S$.
\end{thm}

\begin{proof}
Given $t \in T$, there is some $s\in S$ such that $g(s) = t$; then we let $h(t) := [s]= e_S(s)$.

To see that $h$ is well-defined, let $s, s' \in S$ (we may assume both are nonzero) such that $g(s) = g(s') = t$; we must show that $[s] = [s']$, i.e. $\ann(s) = \ann(s')$.  So let $a\in \ann(s)$. Then $as = 0$, so $g(a) g(s) = g(a) g(s') = 0$, so $as' = 0$, so $a \in \ann(s')$.  By symmetry, $\ann(s) = \ann(s')$.

Next, we show that $h$ is an Armendariz map.  Given $[s] \in E(S)$, we have $h(g(s)) = [s]$, so $h$ is surjective.  We have $h(0) = h(g(0)) = [0] = 0$.  If $h(t) = 0$, let $s \in S$ with $t = g(s)$; then $h(g(s)) = [s] = 0$, so $s=0$, whence $t = g(s) = g(0) = 0$.  If $tt' = 0$, let $s, s' \in S$ with $g(s) = t$, $g(s') = t'$; then $g(s) g(s') = 0$, so $ss' = 0$, so $0 = [ss'] = [s][s'] = h(t) h(t')$.  Finally, if $h(t) h(t') = 0$, this means that $[s][s'] = [ss']= 0$, whence $ss' = 0$, so $g(s)g(s') = tt' = 0$.
\end{proof}

If $R$ is a commutative ring, then $\Gamma(E(R), \cdot) = \Gamma_E(R)$, and these graphs have been studied by \cite{SpWi-eqzd} and others.  On the other hand, if $X$ is a pearled topological space (that is, every closed set contains a closed point; see \S\ref{sec:lattice} and \S\ref{sec:prl}), then $E(\sigma(X), \cap) = (\sigma(Y), \cap)$, where $Y = \Prl(X)$  is the subspace consisting of the closed points of $X$ -- that is, the T$_1$-ification of $X$.  The analogous statement also holds for atomic bounded distributive lattices.

\begin{rmk}
The equivalence class quotient may of course also be defined for semigroups that are not nilpotent-free.  However, no analogue of Theorem~\ref{thm:Armendariz} holds at this level of generality.

To see this, consider the example $R := \F_2[x,y] / (x^2, xy, y^2)$, where $x,y$ are indeterminates and $\F_2$ is the field of two elements.  The zero-divisors are $x$, $y$, and $x+y$, all of which annihilate each other (and themselves), which means that $\Gamma(R)$ is a triangle and $\Gamma_E(R)$ is a single point.  These two graphs have different diameters (1 and 0 respectively), different girths (3 and $\infty$ respectively), different clique numbers (3 and 1 respectively), and different chromatic numbers (3 and 1 respectively).
\end{rmk}

\section{Polynomial and power-series algebras and annihilating-ideal graphs}\label{sec:content}
For a commutative ring $R$, let $\FId R$ (resp. $\CId R$) be the set of finitely generated (resp. countably generated) ideals of $R$; this of course forms a semigroup with zero under multiplication; it is nilpotent-free if and only if $R$ is reduced.  The \emph{content map} (resp. \emph{generalized content map}) $c: R[X] \ra \FId R$ (resp. $c': R[\![X]\!] \ra \CId R$) is the set map that sends each polynomial (resp. power series) to its \emph{content}  (resp. \emph{generalized content}), defined to be the ideal generated by its coefficients.  These maps frequently fail to be a homomorphism, but are often Armendariz maps.

Indeed, Armendariz \cite[Lemma 1]{Arm-ext} showed that whenever $R$ is reduced and $f,g \in R[X]$, then $fg=0$ if and only if $ab=0$ for all coefficients $a$ of $f$ and $b$ of $g$.  In other words (in our terms), the content map is an Armendariz map for such rings.  Analogously, we have the following:
\begin{lemma}
Let $R$ be a reduced ring and $f,g \in R[\![X]\!]$ such that $fg = 0$.  Then $c'(f) c'(g) = 0$, where $c': R[\![X]\!] \ra (\CId R,\cdot)$ is the map that sends a power series to the ideal generated by its coefficients.  Hence, $c'$ is an Armendariz map.
\end{lemma}

\begin{proof}
The fact that $c'(f)c'(g)=0$ when $fg=0$ is given in \cite[Theorem 12.3]{Wat-book}.  To show that the map is Armendariz, all that remains is to note that it is surjective, that $f=0$ if and only if $c'(f) = 0$, and that since $c'(fg) \subseteq c'(f) c'(g)$, it follows that if $c'(f)c'(g) =0$, one has $c'(fg)=0$ and thus $fg=0$.
\end{proof}

A ring where the content map is Armendariz is called an \emph{Armendariz ring}.  On the other hand, a ring in which the content map is a homomorphism is called a \emph{Gaussian ring}, and though it is well-known that not all Armendariz rings are Gaussian, we would like to add the following large class of counterexamples:

\begin{lemma}\label{lem:Arm-not-Gauss}
Let $(R,\m)$ be a quasi-local reduced ring that is not an integral domain.  Then $R$ is not Gaussian.
\end{lemma}

\begin{proof}
Take any pair $a,b \in R \setminus \{0\}$ such that $ab=0$.  Let $X$ be an indeterminate over $R$, and let $f := aX + b$, $g := bX+a \in R[X]$.  We have $c(f) = c(g) = (a,b)$, so that $c(f)c(g) = (a^2, b^2)$.  On the other hand, $fg=(a^2 +b^2)X$, so that $c(fg) = (a^2 + b^2)$.  Suppose that $a^2 \in (a^2 + b^2)$.  That is, there is some $u\in R$ such that $a^2 = (a^2 + b^2) u$, so that \begin{equation}\label{eq:1-uu}
(1-u) a^2 = u b^2.
\end{equation}
If $u \notin \m$, then $u$ is a unit so that $v := (1-u)/u \in R$, and multiplying Equation~\ref{eq:1-uu} by $u^{-1}$ gives \[
v a^2 = b^2.
\]
Thus, $b^4 = v a^2 b^2 = v (ab)^2 = v 0^2 = 0$, so that $b$ is nilpotent, contradicting the fact that $R$ is reduced.

Since the assumption that $u\notin \m$ led to a contradiction, it follows that $u \in \m$.  But then $1-u \notin \m$, whence $1-u$ is a unit, $w=u/(1-u) \in R$, and multiplying both sides of Equation~\ref{eq:1-uu} by $(1-u)^{-1}$ gives \[
a^2 = w b^2.
\]
Thus, $a^4 = a^2 w b^2 = (ab)^2 w = 0^2 w = 0$, whence $a$ is nilpotent, again contradicting the fact that $R$ is reduced.

Hence, $a^2 \in c(f) c(g) \setminus c(fg)$.
\end{proof}

Thus, by using Theorem~\ref{thm:Armendariz} along with what is already known in the literature, one can give relationships between graph-theoretic invariants of the zero-divisor graphs of $R[X]$ and $\FId R$ (resp. $R[\![X]\!]$ and $\CId R$) when $R$ is reduced.  Indeed, much work has already been done on the relationship between invariants of $\Gamma(R)$, $\Gamma(R[X])$, and $\Gamma(R[\![X]\!])$, so the studies may be combined to get relationships between all five zero-divisor graphs.  We choose not to embark on such a systematic study here.  However, as a contribution to the former, we have:

\begin{prop}
Let $R$ be reduced.  Then $\clq{\Gamma(R)} = \clq{\Gamma(R[X])} = \clq{\Gamma(R[\![X]\!])}$ and $\chro(\Gamma(R)) = \chro(\Gamma(R[X])) = \chro(\Gamma(R[\![X]\!]))$.
\end{prop}

\begin{proof}
We define a set map $\beta: \ver{\Gamma(R[\![X]\!])} \ra \ver{\Gamma(R)}$ as follows.  For each $f\in \ver{\Gamma(R[\![X]\!])}$, pick a nonzero element $b\in c'(f)$ and let $\beta(f) = b$.  To see that this is well-defined, note that for any such $f$, there is some nonzero $g\in R[\![X]\!]$ such that $fg = 0$, so that since $c'$ is an Armendariz map, we have $c'(f) c'(g) = 0$, so that in particular $\beta(f) \beta(g) = 0$ and $\beta(f)$ is a nonzero zero-divisor of $R$ as required.

For the clique number claim, first note that any clique in $\Gamma(R)$ is a clique in $\Gamma(R[X])$, and that any clique in $\Gamma(R[X])$ is a clique in $\Gamma(R[\![X]\!])$; hence $\clq{\Gamma(R)} \leq \clq{\Gamma(R[X])} \leq \clq{\Gamma(R[\![X]\!])}$).  Conversely, let $\{f_\lambda \mid \lambda \in \Lambda\}$ be a clique in $\Gamma(R[\![X]\!])$.  Then for any $\lambda \neq \mu$, we have $f_\lambda f_\mu = 0$, whence (since $c'$ is an Armendariz map) $\beta(f_\lambda) \beta(f_\mu) \in c'(f_\lambda) c'(f_\mu) = 0$, and moreover  $\beta(f_\lambda) \neq \beta(f_\mu)$ since $R$ is reduced, so $\{\beta(f_\lambda) \mid \lambda \in \Lambda\}$ forms a clique in $\Gamma(R)$.  Thus $\clq{\Gamma(R[\![X]\!])} \leq \clq{\Gamma(R)}$.

For chromatic number, first note that any coloring of $\Gamma(R[\![X]\!])$ restricts to a coloring of $\Gamma(R[X])$, and any coloring of $\Gamma(R[X])$ restricts to a coloring of $\Gamma(R)$, so that $\chro(\Gamma(R)) \leq \chro(\Gamma(R[X])) \leq \chro(\Gamma(R[\![X]\!]))$.  Conversely, let $q: \ver{\Gamma(R)} \ra A$ be a coloring of $\Gamma(R)$.  Then for any $f, g \in \ver{\Gamma(R[\![X]\!])}$ such that $fg = 0$, we have $\beta(f) \beta(g) = 0$, so that $(q \circ \beta)(f) \neq (q \circ \beta)(g)$, whence $q \circ \beta: \ver{\Gamma(R[\![X]\!])} \ra A$ is a coloring of $\Gamma(R[\![X]\!])$.  Thus, $\chro(\Gamma(R[\![X]\!])) \leq \chro(\Gamma(R))$.
\end{proof}

\begin{rmk}
The invariants of chromatic number and clique number have not been studied as much for the graph $\Gamma(R)$ as for the graph $\Gamma_0(R)$ of Beck \cite{Be-color}.  One lets \emph{all} the elements of $R$ be vertices of $\Gamma_0(R)$, and $a,b$ have an edge when $ab=0$.  If $\Gamma(R) = \emptyset$, then Beck's graph is just a star graph where the connecting vertex is $0$.  If $\Gamma(R) \neq \emptyset$, one obtains $\Gamma_0(R)$ from it by first adjoining $0$ and making an edge between every vertex of $\Gamma(R)$ and $0$ (i.e. taking the \emph{cone over $0$} of $\Gamma_0(R)$), and then connecting all the non-zerodivisors to $0$.

If $R$ is an integral domain, $\chro(\Gamma_0(R)) = \clq{\Gamma_0(R)} = 2$ and $\chro(\Gamma(R)) = \clq{\Gamma(R)} = 0$.

If $R$ is not a domain, then $\chro(\Gamma_0(R))$ is infinite iff $\chro(\Gamma(R))$ is infinite, and if so, they are equal.  Moreover, $\clq{\Gamma_0(R)}$ is infinite iff $\clq{\Gamma(R)}$ is infinite, and if so they are equal.

Finally, $\chro(\Gamma(R))$ is finite but nonzero, then $\chro(\Gamma(R)) = \chro(\Gamma_0(R)) - 1$, and if $0<\clq{\Gamma(R)}$ is finite but nonzero, then $\clq{\Gamma(R)} = \clq{\Gamma_0(R)} -1$.

\begin{proof}
The assertions are clear when $R$ is a domain, so from now on we assume $R$ has nonzero zero-divisors.

For any maximal clique in $\Gamma(R)$, one may augment it with the vertex $0$ to get a maximal clique in $\Gamma_0(R)$.  Conversely, no maximal clique in $\Gamma_0(R)$ includes a non-zerodivisor, and every such clique includes $0$, so one may omit the vertex $0$ to obtain a maximal clique in $\Gamma(R)$ with one fewer vertex.

Now, suppose we have a coloring $c: \ver{\Gamma(R)} \ra A$.  We obtain a coloring of $\Gamma_0(R)$ as follows.  Let $y$ be a non-zero zero-divisor of $R$ and let $d=c(y)$.  Let $z$ be a new color.  Then we extend $c$ to define $c': \ver{\Gamma_0(R)}=R \ra A \cup \{z\}$ by letting $c(0) = z$ and $c(n) = d$ for all non-zerodivisors of $R$.

Conversely, suppose we have a coloring $c: \ver{\Gamma_0(R)} = R \ra B$.  Let $z = c(0)$.  Then clearly $c^{-1}(\{z\}) = \{0\}$, so by restricting to the vertices of $\Gamma(R)$, we may omit $z$ to get a coloring $c': \ver{\Gamma(R)} \ra B \setminus \{z\}$.
\end{proof}
\end{rmk}

The zero-divisor graph of  the semigroup $(\Id R, \cdot)$ was introduced by Behboodi and Rakeei \cite{BehRak-ag1}, who called it the \emph{annihilating-ideal graph} of $R$, denoted $\AG(R)$.   Using our framework of Armendariz maps along with some results of \cite{AnMu-zdgraph}, we prove the following, conjectured by Behboodi and Rakeei \cite[Conjecture 1.11]{BehRak-ag2} and first proved in \cite{AANIS-colorai} via completely different methods from ours.

\begin{thm}\label{thm:AGconj}
Let $R$ be a reduced ring with more than two minimal primes.  Then $\gir \AG(R) = 3$.
\end{thm}

\begin{proof}
First, note that the total ring of quotients of $R$ cannot be a product of two fields, or else the kernels of the corresponding maps from $R$ to the fields would be the two unique minimal primes of $R$.  But according to \cite[Theorems 2.2 and 2.4]{AnMu-zdgraph}, the fact that the total ring of quotients is not a product of two fields implies that $\gir \Gamma(R) = 3$.  Then \cite[Theorem 3.2]{AnMu-zdgraph} implies that $\gir \Gamma(R[X]) = 3$.  But since the content map $c: R[X] \ra (\FId R, \cdot)$ is an Armendariz map, it follows from Theorem~\ref{thm:Armendariz} that $\gir \Gamma(\FId R, \cdot) = 3$.  That is, there are distinct finitely-generated ideals $I, J, K$ of $R$ such that $IJ = IK = JK = 0$.  These also create a 3-cycle in $\AG(R)$, so that $\gir \AG(R) = 3$ as well. 
\end{proof}

\section{$T_1$ topological spaces and subset lattices}\label{sec:lattice}
 Recall that a \emph{lattice} is a partially ordered set $(\cL, \leq)$ such that every pair $x, y\in \cL$ of elements has a unique least upper bound (called their \emph{join} $x \vee y$) and a unique greatest lower bound (called their \emph{meet} $x \wedge y$).  We will assume that all lattices are atomic, bounded, and distributive.  \emph{Bounded} means that there is a greatest element $1$ and a least element $0$.  \emph{Distributive} means that each of the meet and join operations follow the distributive law with respect to the other.  \emph{Atomic} means that for any nonzero element $x$, there is an atom $a$ such that $x\geq a$.  An \emph{atom} of a bounded lattice $\cL$ is an element $a\in \cL$ such that when $x\in \cL$ and $x\leq a$, either $x=a$ or $x=0$.

For generalities on lattice theory, see the book \cite{DaPr-latbook}.

Note that if one considers such a lattice $\cL$ (i.e. an atomic, bounded, distributive lattice, or \emph{a.b.d.l.}) along with its \emph{meet} operation $\wedge$, it is a nilpotent-free semigroup, with $0$ as the absorbing element, since for $x \in \cL$, we have $x \wedge x = x$.

Let $X$ be a topological space, and suppose $X$ has the property that every nonempty \emph{closed} subset $A$ of $X$ contains a \emph{minimal} nonempty closed subset $B$ of $X$.  (That is, $B$ is closed, and there are no nonempty closed subsets of $X$ properly contained in $B$.  We call such a $B$ an \emph{atom}.)  This is then an atomic bounded distributive lattice, where meet is intersection, join is union, $X$ is the multiplicative identity, $\emptyset$ is the absorbing element, and the minimal nonempty closed subsets of $X$ are the atoms.  If $X$ is a $T_1$-space (i.e. every point is closed), then it has this property, and the one-point sets are the atoms.  More generally, we call a space \emph{pearled} if every nonempty closed subset of $X$ contains a closed point.  Then the atoms are the closed points.  If $X$ is a pearled space, let $Y := \Prl(X)$ be the set of all closed points of $X$, with the subspace topology.  \emph{Note that $Y$ is then a $T_1$ space.}  We will explore the notion of pearled space in its own right in \S\ref{sec:prl}.

If $S$ is a set, consider a bounded distributive lattice $\cL$ of subsets of $S$ (that is, a sublattice of $2^S$).  We say that $\cL$ is $T_1$ if for all $s\in S$, $\{s\} \in \cL$.

Indeed, for any a.b.d.l. $\cL$, we construct a $T_1$ subset lattice as follows.  Let $Y$ denote the set of atoms of $\cL$.  We give a sublattice $\cM$ of $2^Y$ as follows.  For any element $e \in \cL$, let $\alpha(e)$ denote the set of atoms bounded above by $e$.  That is, $\alpha(e) := \{y \in Y \mid y \leq e\}$.  Then let $\cM := \{ \alpha(e) \mid e \in \cL\}$.  Then $\alpha(0) = \emptyset$, $\alpha(1) = Y$, $\alpha(e \vee f) = \alpha(e) \cup \alpha(f)$, and $\alpha(e \wedge f) = \alpha(e) \cap \alpha(f)$, so that not only is $\cM$ an atomic bounded distributive subset lattice on $Y$, but $\alpha: (\cL, \wedge) \ra (\cM, \cap)$ is a semigroup homomorphism.  Moreover, $\alpha$ is surjective by construction, and kernel-free because of the atomicity condition on $\cL$.  Thus, $\alpha$ is an Armendariz map.

For a topological space $X$, let $\sigma(X)$ denote its lattice of closed subsets.

Let $X$ be a pearled topological space.  Seen as a lattice, then, $\sigma(X)$ is an a.b.d.l. whose atoms are the closed points.  Then the construction $\alpha$ above, applied to $X$, just identifies the subset $Y \hookrightarrow X$ of closed points of $X$ under the subspace topology.  Of course the resulting map $\alpha: \sigma(X) \ra \sigma(Y)$ (given by intersection with $Y$) is a kernel-free surjective lattice homomorphism, and in particular an Armendariz map.  A good example of this is $X=\Spec R$ (since every ideal is contained in a maximal ideal); then $Y=\Max R$, the maximal ideal spectrum of $R$ in the Zariski topology.  See \S\ref{sec:spec}.

For a set $Y$ and a subset lattice $\cL$ of $2^Y$ (e.g. the closed subset lattice of some topology on $Y$), we say that $\cL$ is \emph{irreducible} if for all $A, B \in \cL \setminus \{Y\}$, we have $A \cup B \neq Y$.  We say that $\cL$ is \emph{connected} if for all $A,B \in \cL \setminus \{Y\}$ such that $A \cap B = \emptyset$, we have $A \cup B \neq Y$.  This terminology is inherited from topology and algebraic geometry.

\begin{prop}\label{pr:charirrconn}
Let $Y$ be a set and let $\cL$ be a $T_1$ sublattice of $2^Y$.  Then $\ver{\Gamma(\cL)} = \cL \setminus \{\emptyset, Y\}$.  Moreover: 
\begin{enumerate}
\item\label{it:charirred} $\cL$ is irreducible $\iff$ every pair of vertices of $\Gamma(\cL)$ admits a path of length 2 between them in $\Gamma(\cL)$.
\item\label{it:charconn} $\cL$ is connected $\iff$ every edge of $\Gamma(\cL)$ is part of a 3-cycle.
\end{enumerate}
\end{prop}

\begin{proof}
For the first statement, we may first assume that $Y$ itself is nonempty.  It is clear that $\emptyset$ and $Y$ cannot be vertices of $\Gamma(\cL)$, and for any proper nonempty subset $A$ of $Y$ such that $A \in \cL$, there exists $y \in Y \setminus A$, and by the $T_1$ property we have $\{y\} \in \cL$, so that $A \wedge \{y\} = A \cap \{y\} = \emptyset$.  Thus, $A$ is a vertex of $\Gamma(\cL)$.

\begin{proof}[Proof of (\ref{it:charirred})]
Suppose $\cL$ is irreducible.  The conditions are vacuous if the zero-divisor graph has fewer than two vertices.  Accordingly, let $A, B \in \cL \setminus \{\emptyset, Y\}$ be distinct elements.  Then $A \cup B \neq Y$, so there exists $y\in Y$ such that $y \in Y \setminus (A \cup B)$.  But by the $T_1$ condition, $\{y\} \in \cL$, and $A \cap \{y\} = B \cap \{y\} = \emptyset$, so we get the path $A$ --- $\{y\}$ --- $B$.

On the other hand, suppose $Y$ is not irreducible.  Then there exist $A, B \in \cL$ such that $A \cup B = Y$.  If there were a path $A$ --- $C$ --- $B$ of length 2 in $\Gamma(\cL)$, then we would have $A \cap C = B \cap C = \emptyset$, so that $C =Y \cap C =  (A \cup B) \cap C= (A \cap C) \cup (B \cap C) = \emptyset \cap \emptyset = \emptyset$, contradicting the fact that $C$ is a vertex of $\Gamma(\cL)$.
\end{proof}

\begin{proof}[Proof of (\ref{it:charconn})]
Suppose $\cL$ is connected.  We may assume that $\cL$ has at least one edge, so let $A$ --- $B$ be an edge in $\Gamma(\cL)$.  Then $A, B \in \cL \setminus \{\emptyset, Y\}$ with $A \cap B = \emptyset$, so by the connectedness property, $A \cup B \neq Y$.  Letting $y \in Y \setminus (A \cup B)$, we get a 3-cycle in $\Gamma(\cL)$ between $A$, $B$, and $\{y\}$. 
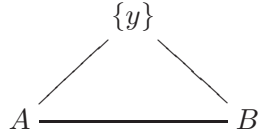
\[\xymatrix{
& \{y\} \ar@{-}[dl] \ar@{-}[dr]\\
A \ar@{-}[rr] & & B
}\]

On the other hand, suppose $\cL$ is not connected, and let $A$, $B$ form a disconnection of $\cL$.  That is, $A, B \in \cL \setminus \{\emptyset, Y\}$, $A \cap B = \emptyset$, and $A \cup B = Y$.  Then $A$ --- $B$ is an edge in $\Gamma(\cL)$, but if $C$ formed edges with both $A$ and $B$, then we would have $C = \emptyset$ by the same argument as in the irreducible case, a contradiction.
\end{proof}

\end{proof}

\begin{thm}\label{thm:chroclq}
Let $Y$ be a set and let $\cL$ be a $T_1$ sublattice of $2^Y$.  Then $\chro(\Gamma(\cL)) = \clq \Gamma(\cL) = $ the cardinality of $Y$.
\end{thm}

\begin{proof}
First recall that for any graph $G$, one has $\chro(G) \geq \clq G$.  This is because in order to color a clique of cardinality $\lambda$, one needs $\lambda$ colors, since each vertex has an edge with every other vertex.  So $\chro(\Gamma(\cL)) \geq \clq \Gamma(\cL)$.

Next, we show that $\clq \Gamma(\cL) \geq |Y|$.  Indeed, let $G$ be the induced subgraph of $\Gamma(\cL)$ such that $\ver G = \{\{y\} \mid y \in Y\}$.  All these singletons are vertices of $\Gamma(\cL)$ by the $T_1$ property, and of course $\{y\} \cap \{z\} = \emptyset$ whenever $y\neq z$, so $G$ contains the edge $\{y\}$ --- $\{z\}$ and hence is a $|Y|$-clique in $\Gamma(\cL)$.

Finally, we show that $\chro(\Gamma(\cL)) \leq |Y|$.  We define a set map $c: \ver{\Gamma(\cL)} \ra Y$ as follows: for any $A \in \cL \setminus \{\emptyset, Y\}$, pick an element $x\in A$, and let $c(A) = x$.  (Such a map is a direct application of the axiom of choice.)  Then for any edge $A$ --- $B$ in $\Gamma(\cL)$, we have $A \cap B = \emptyset$, so that since $c(A) \in A$ and $c(B) \in B$, we have $c(A) \neq c(B)$.  Thus, $c$ is a $|Y|$-coloring of $\Gamma(\cL)$.
\end{proof}

\begin{thm}\label{thm:T1}
Let $Y$ be a set and let $\cL$ be a $T_1$ sublattice of $2^Y$. \begin{enumerate}
\item\label{it:finite} If $Y$ is finite, then $\cL$ is the closed subset lattice of the discrete space structure on $Y$.  That is, $\cL = 2^Y$.  (Hence, $\Gamma(\cL)$ is the complement of the incidence graph on the proper nonempty subsets of $Y$.)
\item\label{it:leq1} If $\#(Y) \leq 1$ then $\Gamma(\cL)$ is empty.
\item\label{it:2} If $\#(Y) = 2$, then $\Gamma(\cL)$ consists of one edge with two vertices connecting them.  Hence, \begin{itemize}
 \item $\diam \Gamma(\cL) = 1$,
 \item $\gir \Gamma(\cL) = \infty$.
 \end{itemize}
\item\label{it:irred} If $\#(Y) \geq 3$ and $\cL$ is irreducible, then \begin{itemize}
 \item $Y$ is infinite,
 \item $\diam \Gamma(\cL) = 2$,
 \item $\gir \Gamma(\cL) = 3$.
  \end{itemize}
\item\label{it:notirr} If $\#(Y) \geq 3$ and $\cL$ is not irreducible, then \begin{itemize}
 \item $\diam \Gamma(\cL) = 3$, 
 \item $\gir \Gamma(\cL) = 3$.
 \end{itemize}
\end{enumerate}
\end{thm}

\begin{proof}[Proof of (\ref{it:finite})]
Let $A$ be any nonempty subset of $Y$.  Then $A = \bigcup_{a \in A} \{a\}$.  But by the $T_1$ property $\{a\} \in \cL$, and since the union is finite and $\cL$ is a lattice, it follows that $A \in \cL$.
\end{proof}

\begin{proof}[Proof of (\ref{it:leq1})]
Obvious.
\end{proof}

\begin{proof}[Proof of (\ref{it:2})]
If $a,b$ are the elements of $Y$, then $\cL = \{\emptyset, \{a\}, \{b\}, Y\}$, so that the vertices of $\Gamma(\cL)$ are $\{a\}$ and $\{b\}$, which have an edge between them because $a\neq b$.  The rest is clear.
\end{proof}

\begin{proof}[Proof of (\ref{it:irred})]
The infiniteness follows from part (\ref{it:finite}).  By Proposition~\ref{pr:charirrconn}, any pair of vertices is connected by a path of length 2, so $\diam \Gamma(\cL) \leq 2$.  To see that the diameter equals 2, let $y,z \in Y$ be distinct elements; then $\{y\}, \{y,z\} \in \cL$ but they have nonempty intersection, so that $d(\{y\}, \{y,z\}) \geq 2$, and so $\diam \Gamma(\cL) = 2$.  As for the girth, if $x, y, z$ are distinct elements of $Y$, then the vertices $\{x\}, \{y\}, \{z\}$ of $\Gamma(\cL)$ form a 3-cycle.
\end{proof}

\begin{proof}[Proof of (\ref{it:notirr})]
For the girth statement, if $a,b,c$ are distinct elements of $Y$, then $\{a\}, \{b\}, \{c\} \in \cL$ by the $T_1$ property, and they form a 3-cycle since the sets are pairwise disjoint.

For the diameter statement, we break our analysis into two cases, based on whether $\cL$ is connected or not.

First suppose $\cL$ is connected.  Since $\diam \Gamma(\cL) \leq 3$, to show equality we need only exhibit a pair of vertices whose distance from each other is greater than two.  Let $A$, $B$ be distinct vertices of $\Gamma(\cL)$ such that $A \cup B = Y$.  Since $\cL$ is connected, $A \cap B \neq \emptyset$, so $d(A, B) \geq 2$.  However, given an edge $A$ --- $C$ in $\Gamma(\cL)$, we have $C \cap A  = \emptyset$, from which it follows that $C \subseteq B$, so that $C \cap B \neq \emptyset$.  Thus, there is no path of length 2 from $A$ to $B$, whence $d(A,B) = 3$, so that $\diam \Gamma(\cL) = 3$.

If $\cL$ is not connected, let $A$, $B$ be a disconnection of $\cL$.  That is, $A, B \in \cL \setminus \{\emptyset, Y\}$, $A \cap B = \emptyset$, and $A \cup B = Y$.  Since $Y$ has at least three elements, without loss of generality we have $|A| \geq 2$.  Let $x \in A$.  Then $\{x\} \in \cL$ by the $T_1$ property.  Then by the lattice property, $C := \{x\} \cup B \in \cL$, and since $C \notin \{\emptyset, Y\}$, we have that $C$ is a vertex of $\Gamma(\cL)$.  We claim that $d(A,C) =3$.  We have $A \neq C$ (so that $d(A,C) \neq 0$) and $A \cap C = \{x\} \neq \emptyset$ (so that $d(A,C) \neq 1$).  Now take an arbitrary edge of $\Gamma(\cL)$ that has $A$ as a vertex, say $A$ --- $D$.  Then $A \cap D = \emptyset$, which means that $D \subseteq Y \setminus A \subseteq B \subseteq C$, whence $D \cap C \neq \emptyset$.  This shows that $d(A, C) \neq 2$, so we must have $d(A,C) = 3$.  Thus, $\diam \Gamma(\cL)=3$.
\end{proof}

\section{Pearled spaces}\label{sec:prl}

In \S\ref{sec:lattice}, we introduced the notion of a \emph{pearled} topological space, meaning a space where every nonempty closed subset contains a closed point.  Examples abound (e.g. all $T_1$ spaces and $\Spec R$ for any commutative ring $R$).  Moreover, as remarked earlier in \S\ref{sec:lattice}, when $X$ is such a space, there is then an Armendariz map from the closed subset lattice of $X$ to that of the closed points of $X$ (indeed something of a canonical one, \emph{cf.} \S\ref{sec:eq}), which is why this axiom is useful for our purposes.  However, as the present paper seems to be the first place where this property appears, we felt it to be the proper venue to inquire into its relationship with other, well-known properties, in order to facilitate further investigations. Accordingly, we devote the present section to this notion, comparing it to other separation axiom-type properties a topological space may have.  However, nothing that is done in this section affects the rest of the paper, so the reader may skip it without losing the flow.

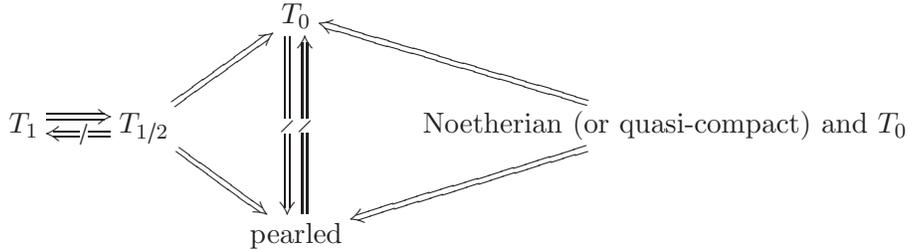
\[ \xymatrix{
& & T_0 \ar@<-.7ex>@{=>}[dd]|{\neswline} \\
T_1 \ar@<+.7ex>@{=>}[r] & T_{1/2} \ar@<+.7ex>@2{->}[l]|{/} \ar@{=>}[ur] \ar@{=>}[dr] & & \text{Noetherian (or quasi-compact) and } T_0 \ar@{=>}[ul] \ar@{=>}[dl]\\
& & \text{pearled}  \ar@<-.7ex>@{=>}[uu]|{\neswline}
}\]
Recall: $T_1$ means all singleton subsets are closed.  $T_{1/2}$ (first defined in \cite{Lev-T1/2}; see also \cite{Dun-T1/2}) means that every singleton subset is either open or closed.  $T_0$ means that for any distinct pair $x,y$ of points, $\overline{\{x\}} \neq \overline{\{y\}}$.
 A space is \emph{Noetherian} if every strictly descending chain of closed subsets is of finite length (e.g. any finite space is Noetherian, and the spectrum of a Noetherian ring is Noetherian).

It is obvious that any $T_1$ space is $T_{1/2}$.

To see that $T_{1/2} \implies T_0$ (known, but included here for completeness), let $X$ be a $T_{1/2}$-space, and let $x,y$ be two distinct points such that $x\in \overline{\{y\}}$.  Then in particular $\{y\}$ is not closed, whence it is open, so that $X \setminus \{y\}$ is closed, so that since the latter set contains $x$, it follows that $\overline{\{x\}} \subseteq X \setminus \{y\}$.

To see that $T_{1/2} \implies$ pearled, let $X$ be a $T_{1/2}$-space.  Let $A$ be a nonempty closed subset of $X$, and suppose $A$ contains no closed points.  Then every point of $A$ is open.  So let $x\in A$.  Then $A \setminus \{x\} = \bigcup_{\overset{y \in A}{y\neq x}} \{y\}$ is open, whence $\{x\}$ is closed, contradicting our assumption.

Now suppose $X$ is a Noetherian\footnote{A similar argument allows us to replace the Noetherian condition with the weaker assumption that the space is quasi-compact.  However, one must appeal to transfinite induction.} $T_0$-space.  We claim that every non-singleton closed subset of $X$ contains a proper closed subset.  Let $A$ be a non-singleton closed subset of $X$, and let $x,y \in A$ be distinct points.  Without loss of generality, $y \notin \overline{\{x\}}$, by the $T_0$ property.  Hence $\overline{\{x\}}$ is a proper closed subset of $A$, proving the claim.  Now let $B=B_0$ be a closed subset of $X$, and suppose $B$ does not contain a closed point.  Then by the claim, $B_0$ contains a proper closed subset $B_1$, which in turn does not contain a closed point, so that it contains a proper closed subset $B_2$, and so forth to make an infinite strictly descending chain of closed subsets of $X$.  But this violates the Noetherian condition.

To see that $T_{1/2} \not \implies T_1$ (known, but included here for completeness), use the Sierpi\'nski space (that is, the prime spectrum of a rank 1 discrete valuation ring) as a counterexample.

To see that pearled $\not \implies T_0$, let $X = \{a,b,c\}$, and topologize it by declaring the closed sets to be $\emptyset$, $\{c\}$, and $X$.  This is clearly a pearled topological space, but it is not $T_0$ since $\overline{\{a\}} = \overline{\{b\}}$.

To see that $T_0 \not \implies$ pearled, let $X = \N_0$, the nonnegative integers, and topologize it by declaring the closed sets to be $\emptyset$ and all intervals of the form $[n, \infty)\cap \N_0$ for $n\in \N_0$.  Then $X$ is $T_0$, since if $m<n$ then $m \notin \overline{\{n\}} = [n,\infty)\cap \N_0$, but it is not pearled, because there are no closed points at all.

\section{Prime spectra, maximal spectra, and $\USpec R$}\label{sec:spec}
We let $\Spec R$ denote the prime spectrum of a ring with the Zariski topology.  $\USpec R$ will denote the same set, but topologized differently.
 \begin{defn}
We let $\USpec R$ be the topological space whose elements are the prime ideals of $R$ and whose closed sets are the sets of the form $\bigcup_{\lambda \in \Lambda} V(I_\lambda)$, where $\{I_\lambda \mid \lambda \in \Lambda\}$ is an arbitrary set of ideals of $R$.
\end{defn}

Note that $\USpec R$ is the smallest refinement of $\Spec R$ that makes the closed set lattice complete.  That is, every \emph{union} of closed sets in $\Spec R$ is a closed set in $\USpec R$.  In other words, $\USpec R$ is an \emph{Alexandroff space}, hence the notation. $\Max R$ denotes the set of maximal ideals of $R$ in the Zariski topology.  (That is, we consider it as a subspace of $\Spec R$.)  $\Jac R$ will denote the intersection of all maximal ideals (i.e. the \emph{Jacobson radical}) of $R$.

\begin{rmk}
The reader may wonder why we consider two different topologies on the prime spectrum of a ring.  The reason will appear in \S\ref{sec:tensor}, where among other things, we utilize the concept of \emph{locally Nakayama modules}, whose set of supports is exactly the closed sets in $\USpec R$.
\end{rmk}

The map $(\sigma(\Spec R), \cap) \ra (\sigma(\Max R), \cap)$ that sends $C \mapsto C \cap \Max R$ is a kernel-free surjective semigroup homomorphism, hence an Armendariz map, as we have noted in the introductory paragraphs of \S\ref{sec:lattice}.  On the other hand, the map $(\sigma(\USpec R), \cap) \ra (2^{\Max R}, \cap)$ is likewise a kernel-free surjective homomorphism.  To see this, note that since $\{\m\}$ is closed in $\Spec R$ for any $\m \in \Max R$, it follows that any subset of $2^{\Max R}$ is closed in $\USpec R$; then apply the previous remarks on pearled spaces.

Thus, we get the following theorem:

\begin{thm}\label{thm:specs}
Let $R$ be a commutative ring.  Consider the two zero-divisor graphs $G:=\Gamma(\sigma(\Spec R))$ and $H:= \Gamma(\sigma(\USpec R))$.
\begin{enumerate}
\item\label{it:spchro} $\chro(G) = \chro(H) = \clq G = \clq H = |\Max R|$.
\item\label{it:sploc} If $R$ is local (i.e. has only one maximal ideal), then $G = H = \emptyset$.
\item\label{it:sp2loc} Suppose $R$ has exactly two maximal ideals $\m_1, \m_2$.  \begin{itemize}
 \item If all nonmaximal prime ideals are contained in $\m_1 \cap \m_2$, then $\diam G = \diam H = 1$.  Otherwise $\diam G = \diam H = 2$.
 \item If there exist $\p_1, \p_2 \in \Spec R \setminus \Max R$ such that $\p_i \subseteq \m_i$ for $i=1,2$ but $\p_1 \nsubseteq \m_2$ and $\p_2 \nsubseteq \m_1$, then $\gir G = \gir H = 4$; otherwise $\gir G = \gir H = \infty$.
 \end{itemize}
\item\label{it:spdiamgir} If $|\Max R| \geq 3$, then $\diam H = \gir H = \gir G = 3$.
\item\label{it:spJacp} If $|\Max R| \geq 3$ and $\Jac R$ is prime, then $\diam G = 2$.
\item\label{it:spJacnp} If $|\Max R| \geq 3$ and $\Jac R$ is not prime, then $\diam G = 3$.
\end{enumerate}
\end{thm}

\begin{proof}
First, we identify which  nonempty closed subsets of $\Spec R$ (resp. $\USpec R$) are zero-divisors.  Since the maps are Armendariz, a closed subset $C$ is a zero-divisor in $\Gamma(\sigma(\Spec R))$ (resp. in $\Gamma(\sigma(\USpec R))$) if and only if its image is a zero-divisor in $\Gamma(\sigma(\Max R))$ (resp. in $\Gamma(2^{\Max R})$.  But this holds if and only if the image of $C$ is a proper nonempty subset of $\Max R$  -- that is, $C$ does not contain all of the maximal ideals of $R$. 

Next, we use the fact that the maps $(\sigma(\Spec R), \cap) \ra (\sigma(\Max R), \cap)$ and $(\sigma(\USpec R), \cap) \ra (2^{\Max R}, \cap)$ given above are Armendariz maps, along with Theorems~\ref{thm:Armendariz}, \ref{thm:chroclq}, and \ref{thm:T1}.  The point is that both $(\sigma(\Max R), \cap)$ and $(2^{\Max R}, \cap)$ are $T_1$ subset lattices of $2^{\Max R}$.

Parts~(\ref{it:spchro}) and (\ref{it:sploc}) then follow immediately.  For part~(\ref{it:spdiamgir}), it is enough to note that $(2^Y, \cap)$ is never irreducible if $|Y| \geq 2$, since one may simply partition the set $Y$ into two disjoint nonempty subsets; here we use $Y=\Max R$

 Now, suppose $\Max R = \{\m_1, \m_2\}$ is a doubleton (i.e. the situation of part~(\ref{it:sp2loc})).  We claim first that the corresponding Armendariz map of semigroups is a bijection if and only if all nonmaximal primes are contained in $\m_1 \cap \m_2$.  Let $X := \Spec R$ or $\USpec R$.  If there is some nonmaximal prime $\p \nsubseteq \m_1 \cap \m_2$, then we may assume $\p \subset \m_1$ and $\p \nsubseteq \m_2$.  Then $V(\p) \cap \Max R = \{\m_1\}$, so that $C := V(\p)$ is a zero-divisor, but $C \cap \{\m_1\} \neq \emptyset$, so that $d(C, \{\m_1\}) \neq 1$, whence $\diam \Gamma(\sigma(X)) = 2$.  On the other hand, if all nonmaximal primes are contained in $\m_1 \cap \m_2$, then $\{\m_1\}$, $\{\m_2\}$ are the only zero-divisors of $(\sigma(X), \cap)$, since if $C$ is a nonempty closed subset and some nonmaximal prime $\p \in C$, then since $\p \subseteq \m_1 \cap \m_2$, both maximal ideals are in $C$, so $C$ is a non-zerodivisor.  Therefore, the induced map on zero-divisor graphs is a bijection, and we have $\diam \Gamma(\sigma(X)) = 1$ and $\gir \Gamma(\sigma(X)) = \infty$.

However, if all nonmaximal primes not contained in $\m_1 \cap \m_2$ are in $\m_1$ (resp. in $\m_2$), then similar considerations show that $\Gamma(\sigma(X))$ is a star graph with $\{\m_2\}$ (resp. $\{\m_1\}$) in the middle, so that $\gir \Gamma(\sigma(X)) = \infty$.  But if there exist  $\p_1$, $\p_2$ as in the statement of the theorem, then $\gir \Gamma(X) = 4$ because the following 4-cycle is a subgraph of $\Gamma(\sigma(X))$: \[\xymatrix{
V(\p_1) \ar@{-}[d] \ar@{-}[r] & V(\p_2) \ar@{-}[d] \\
\{\m_2\} \ar@{-}[r] & \{\m_1\}
}\]

Finally, note that $\Max R$ is irreducible as a topological space if and only if $\Jac R$ is prime, as is well-known (see \cite[Theorem 3.10(b)]{Ke-cabook}) and may be easily seen directly.  Then parts~(\ref{it:spJacp}) and (\ref{it:spJacnp}) then follow immediately.
\end{proof}

\section{Application to tensor-product semigroups}\label{sec:tensor}

We recall the definition and give some basic properties of locally Nakayama modules:

\begin{defn}\cite{NasPa-vfzd}
Let $R$ be a commutative ring.  An $R$-module $M$ is said to be \emph{locally Nakayama} if  for all $\p \in \Spec R$, $M_\p \neq 0$ implies that $M_\p \neq \p M_\p$.  The class of locally Nakayama $R$-modules up to isomorphism is denoted by ${}_R\NAK$.
\end{defn}

In other words, $M$ is locally Nakayama iff for all $\p \in \Supp M$, $M \otimes_R \kappa(\p) \neq 0$, where $\kappa(\p) = R_\p / \p R_\p$ is the residue field of $\p$.

\begin{prop}
Let $R$ be a commutative ring.
\begin{enumerate}
\item\label{it:NAKfg} ${}_R\md \subseteq {}_R\NAK$.
\item\label{it:NAKtensor} If $M, N \in {}_R\NAK$, then $M \otimes_R N \in {}_R\NAK$ and $\Supp_R(M \otimes_R N) = \Supp_RM \cap \Supp_RN$.
\item\label{it:NAKext} ${}_R\NAK$ is closed under extensions.
\item\label{it:NAKsum} ${}_R\NAK$ is closed under \emph{arbitrary} direct sums.
\item\label{it:NAKuspec} $\sigma(\USpec R) = \{\Supp_RM \mid M \in {}_R\NAK\}$.
\end{enumerate}
\end{prop}

\begin{proof}
Parts (\ref{it:NAKfg}) and (\ref{it:NAKtensor}) are contained within \cite[Theorem 4]{NasPa-vfzd}.

To see (\ref{it:NAKext}), let $0 \ra L \ra M \ra N \ra 0$ be a short exact sequence of $R$-modules such that $L, N \in {}_R\NAK$.  Let $\p \in \Supp M = \Supp L \cup \Supp N$.  If $\p \in \Supp N$, then $N \otimes_R \kappa(\p) \neq 0$, so from right-exactness of tensor, it follows that $M \otimes_R \kappa(\p) \neq 0$.  On the other hand, if $\p \notin \Supp N$, then $N_\p = 0$, so since localization is exact, it follows that $L_\p \cong M_\p$, whence $M_\p / \p M_\p \cong L_\p / \p L_\p \neq 0$.

Note that (\ref{it:NAKsum}) was proved in \cite[Theorem 4]{NasPa-vfzd} for \emph{finite} direct sums.  To see it for arbitrary direct sums, let $\{M_\lambda \mid \lambda \in \Lambda\}$ be a set of locally Nakayama modules and $M := \bigoplus_{\lambda \in \Lambda} M_\lambda$.  If $M_\p \neq 0$, then obviously there is some $\mu \in \Lambda$ with $(M_\mu)_\p \neq 0$, so that $M_\mu \otimes_R \kappa(\p) \neq 0$.  But $M_\mu \otimes_R \kappa(\p)$ is a direct summand of $M \otimes_R \kappa(\p)$, since finite tensor product commutes with arbitrary direct sum, whence $M \otimes_R \kappa(\p) = \bigoplus_{\lambda \in \Lambda} (M_\lambda \otimes_R \kappa(\p))$.  In particular, $M \otimes_R \kappa(\p) \neq 0$.

Finally, we show (\ref{it:NAKuspec}).  Given $C \in \sigma(\USpec R)$, there is a set $\{I_\lambda \mid \lambda \in \Lambda\}$ of ideals of $R$ such that $C = \bigcup_{\lambda \in \Lambda} V(I_\lambda)$.  Let $M := \bigoplus_{\lambda \in \Lambda} (R/I_\lambda)$.  Then $M$ is locally Nakayama by (\ref{it:NAKfg}) and (\ref{it:NAKsum}), and $\Supp M = \bigcup_{\lambda\in \Lambda} \Supp(R/I_\lambda) = C$.

The reverse inclusion is a general fact about supports of $R$-modules.  Namely, let $M$ be \emph{any} $R$-module, let  $\p \in \Supp M$ and $\q \in \Spec R$ such that $\p \subseteq \q$.  Then we have $0 \neq M_\p \cong (M_\q)_\p$, from which it follows that $M_\q \neq 0$, so that $\q \in \Supp M$.  This shows that for any $\p \in \Supp M$, one has $V(\p) \subseteq \Supp M$.  That is, $\Supp M = \bigcup_{\p \in \Supp M} V(\p)$, so that $\Supp M$ is a union of $(\Spec R)$-closed sets, whence $\Supp M \in \sigma(\USpec R)$.
\end{proof}

\begin{thm}
Let $R$ be a commutative ring.  Consider the two zero-divisor graphs $G':=\Gamma({}_R\md, \otimes_R)$ and $H':= \Gamma({}_R\NAK, \otimes_R)$.  
\begin{enumerate}
\item $\chro(G') = \chro(H') = \clq G' = \clq H' = |\Max R|$.
\item If $R$ is local (i.e. has only one maximal ideal), then $G' = H' = \emptyset$.
\item Suppose $R$ has exactly two maximal ideals $\m_1, \m_2$.  Then \begin{itemize}
 \item $\diam G' = \diam H' = 2$ and
 \item $\gir G' = \gir H' = 4$.
 \end{itemize}
\item If $|\Max R| \geq 3$, then $\diam H' = \gir H' = \gir G' = 3$.
\item If $|\Max R| \geq 3$ and $\Jac R$ is prime, then $\diam G' = 2$.
\item If $|\Max R| \geq 3$ and $\Jac R$ is not prime, then $\diam G' = 3$.
\end{enumerate}
\end{thm}

\begin{proof}
After noting the kernel-free surjective homomorphisms (hence Armendariz maps) $({}_R\md, \otimes_R) \arrow{\Supp} (\sigma(\Spec R), \cap)$ and $({}_R\NAK, \otimes_R) \arrow{\Supp} (\sigma(\USpec R), \cap)$, most of the above follows from Theorems~\ref{thm:Armendariz} and \ref{thm:specs}.  All that remains is what happens when $R$ has exactly two maximal ideals $\m_1, \m_2$.  In this case, we note simply that the finitely generated (hence locally Nakayama) modules $R/\m_i$ and $(R/\m_i) \oplus (R/\m_i)$ both map to the set $\{\m_i\}$, respectively for each $i=1,2$.
\end{proof}

\begin{rmk}
At this point, it is natural to ask whether this framework can be extended to the semigroup of \emph{all} $R$-modules $({}_R\Mod, \otimes_R)$ (up to isomorphism, of course).  The problem is that the support map, while surjective (onto $\sigma(\USpec R)$) and kernel-free, is no longer a homomorphism, nor even an Armendariz map.  Indeed, one cannot go very far outside the class of locally Nakayama modules before problems occur.  To see this, let $R$ be any ring of positive Krull dimension, let $\p \subset \q$ be  distinct prime ideals, let $L := \kappa(\p)  = R_\p / \p R_\p$ and $M := R/\q$.  Then  $\Supp_R L = V(\p)$ and $\Supp_R(M) = V(\q)$, so that  $\Supp_R L \cap \Supp_R M = V(\q) \neq \emptyset$, but $L \otimes_R M \cong \kappa(\p) / \q \kappa(\p) = 0$, so $\Supp_R (L \otimes_R M) = \emptyset$.

Thus, it was natural to bring in locally Nakayama modules as a large (conjecturally maximal) semigroup of $R$-modules on which the support map is Armendariz.
\end{rmk}

\section{Application to the comaximal ideal graph}\label{sec:comax}

In this section, we consider the semigroup $\Id(R)$ of ideals of $R$, but under \emph{addition}.  This is also a ``semigroup with zero''; its absorbing element is the unit ideal $R$.  There is an isomorphism between $(\Id(R), +)$ and the semigroup of cyclic modules up to isomorphism along with tensor product over $R$, given by $I \mapsto R/I$.  This is related to the so-called ``co-maximal graphs'' introduced (without a name) by \cite{BhSh-comaxg}\footnote{Here and in the bibliography, we adhere to the convention that the authors of a mathematical paper are listed in alphabetical order.}, and especially its ``retract'' $\Gamma_r(R)$ studied in \cite{LWYY-comaxg}, which embeds as an induced subgraph of the graph $G''$ considered in the following theorem.

\begin{thm}
Let $R$ be a commutative ring.  Consider the graph $G'' = \Gamma(\Id(R), +)$. \begin{enumerate}
\item $\chro(G'') = \clq{G''} = |\Max R|$.
\item If $R$ is local (i.e. has only one maximal ideal), then $G''= \emptyset$.
\item Suppose $R$ has exactly two maximal ideals $\m_1, \m_2$.  \begin{itemize}
 \item If all nonmaximal ideals are contained in $\m_1 \cap \m_2$, then $\diam G''  = 1$.  Otherwise $\diam G''  = 2$.
 \item If there exist $J_1, J_2 \in \Id(R) \setminus \Max R$ such that $J_i \subseteq \m_i$ for $i=1,2$ but $J_1 \nsubseteq \m_2$ and $J_2 \nsubseteq \m_1$, then $\gir G'' = 4$; otherwise $\gir G'' = \infty$.
 \end{itemize}
\item If $|\Max R| \geq 3$, then $\gir {G''} = 3$.
\item If $|\Max R| \geq 3$ and $\Jac R$ is prime, then $\diam {G''} = 2$.
\item If $|\Max R| \geq 3$ and $\Jac R$ is not prime, then $\diam {G''} = 3$.
\end{enumerate}
\end{thm}

\begin{proof}
Note that the map $(\Id R, +) \arrow{V(-)} (\sigma(\Spec R), \cap)$, which sends an ideal to the set of prime ideals containing it, is a kernel-free surjective homomorphism, hence an Armendariz map.  Then the proof is essentially identical to the parts of the proof of Theorem~\ref{thm:specs} concerning the graph $G$ of that theorem.
\end{proof}

\newcommand{\etalchar}[1]{$^{#1}$}
\providecommand{\bysame}{\leavevmode\hbox to3em{\hrulefill}\thinspace}
\providecommand{\MR}{\relax\ifhmode\unskip\space\fi MR }
\providecommand{\MRhref}[2]{%
  \href{http://www.ams.org/mathscinet-getitem?mr=#1}{#2}
}
\providecommand{\href}[2]{#2}


\begin{thebibliography}{LWYY11}

\bibitem[AAN{\etalchar{+}}11]{AANIS-colorai}
Ghodratollah Aalipour, Saieed Akbari, Reza Nikandish, Mohammad Nikmehr, and
  Farzad Shaveisi, \emph{On the coloring of the annihilating-ideal graph of a
  commutative ring}, Discrete Math. (2011), {doi}:10.1016/j.disc.2011.10.020.

\bibitem[AL99]{AnLi-zd}
David~F. Anderson and Philip~S. Livingston, \emph{The zero-divisor graph of a
  commutative ring}, J. Algebra \textbf{217} (1999), no.~2, 434--447.

\bibitem[AM07]{AnMu-zdgraph}
David~F. Anderson and Shashikant~B. Mulay, \emph{On the diameter and girth of a
  zero-divisor graph}, J. Pure Appl. Algebra \textbf{210} (2007), no.~2,
  543--550.

\bibitem[Arm74]{Arm-ext}
Efraim~P. Armendariz, \emph{A note on extensions of {B}aer and {P}.{P}.-rings},
  J. Austral. Math. Soc. \textbf{18} (1974), 470--473.

\bibitem[Bec88]{Be-color}
Istv{\'a}n Beck, \emph{Coloring of commutative rings}, J. Algebra \textbf{116}
  (1988), no.~1, 208--226.

\bibitem[BR11a]{BehRak-ag1}
Mahmood Behboodi and Zahra Rakeei, \emph{The annihilating-ideal graph of
  commutative rings {I}}, J. Algebra Appl. \textbf{10} (2011), no.~4, 727--739.

\bibitem[BR11b]{BehRak-ag2}
\bysame, \emph{The annihilating-ideal graph of commutative rings {II}}, J.
  Algebra Appl. \textbf{10} (2011), no.~4, 741--753.

\bibitem[BS95]{BhSh-comaxg}
S.~M. Bhatwadekar and Pramod~K. Sharma, \emph{A note on graphical
  representation of rings}, J. Algebra \textbf{176} (1995), no.~1, 124--127.

\bibitem[DMS02]{DMS-zdsemi}
Frank~R. DeMeyer, Thomas McKenzie, and Kim Schneider, \emph{The zero-divisor
  graph of a commutative semigroup}, Semigroup Forum \textbf{65} (2002), no.~2,
  206--214.

\bibitem[DP02]{DaPr-latbook}
Brian~A. Davey and Hilary~A. Priestley, \emph{Introduction to lattices and
  order}, second ed., Cambridge Univ. Press, New York, 2002.

\bibitem[Dun77]{Dun-T1/2}
William Dunham, \emph{{$T_{1/2}$}-spaces}, Kyungpook Math. J. \textbf{17}
  (1977), no.~2, 161--169.

\bibitem[Kem11]{Ke-cabook}
Gregor Kemper, \emph{A course in commutative algebra}, Graduate Texts in
  Mathematics, vol. 256, Springer, Heidelberg, 2011.

\bibitem[Lev70]{Lev-T1/2}
Norman Levine, \emph{Generalized closed sets in topology}, Rend. Circ. Mat.
  Palermo (2) \textbf{19} (1970), 89--96.

\bibitem[LWYY11]{LWYY-comaxg}
Dancheng Lu, Tongsuo Wu, Meng Ye, and Houyi Yu, \emph{On graphs related to the
  co-maximal ideals of a commutative ring}, {arXiv}:1106.0072v1, June 2011.

\bibitem[Mul02]{Mu-cszd}
Shashikant~B. Mulay, \emph{Cycles and symmetries of zero-divisors}, Comm.
  Algebra \textbf{30} (2002), no.~7, 3533--3558.

\bibitem[NP10]{NasPa-vfzd}
Peyman Nasehpour and Shiroyeh Payrovi, \emph{Modules having very few
  zero-divisors}, Comm. Algebra \textbf{38} (2010), no.~9, 3154--3162.

\bibitem[SW11]{SpWi-eqzd}
Sandra Spiroff and Cameron Wickham, \emph{A zero divisor graph determined by
  equivalence classes of zero divisors}, Comm. Algebra \textbf{39} (2011),
  no.~7, 2338--2348.

\bibitem[Wat07]{Wat-book}
John~J. Watkins, \emph{Topics in commutative ring theory}, Princeton University
  Press, Princeton, NJ, 2007.

\end{thebibliography}
\end{document}